\documentclass[kpfonts]{patmorin}
\usepackage{pat}
\usepackage{paralist,graphicx,float,anyfontsize}
\usepackage{array,longtable}

\usepackage[utf8]{inputenc}
\usepackage{todonotes}

\usepackage[noabbrev,capitalise]{cleveref}
\crefname{lem}{Lemma}{Lemmas}
\crefname{thm}{Theorem}{Theorems}
\crefname{cor}{Corollary}{Corollaries}
\crefname{prop}{Proposition}{Propositions}
\crefname{conj}{Conjecture}{Conjectures}
\crefname{open}{Open Problem}{Open Problems}
\crefname{obs}{Observation}{Observations}

\crefformat{equation}{(#2#1#3)}
\Crefformat{equation}{Equation #2(#1)#3}

\usepackage[numbers,sort&compress,longnamesfirst]{natbib}

\DeclareMathOperator{\sn}{sn}
\DeclareMathOperator{\qn}{qn}
\DeclareMathOperator{\sqn}{sqn}
\DeclareMathOperator{\dsn}{dsn}

\renewcommand{\le}{\leqslant}
\renewcommand{\leq}{\leqslant}
\renewcommand{\ge}{\geqslant}
\renewcommand{\geq}{\geqslant}

\newcommand{\CartProd}{\mathbin{\square}}

\title{\MakeUppercase{Stack-Number is not Bounded by Queue-Number}}

\author{%
	Vida Dujmovi\'c,\!\!%
	\thanks{School of Computer Science and Electrical Engineering,
		University of Ottawa, Ottawa, Canada (\texttt{vida.dujmovic@uottawa.ca}).
		Research supported by NSERC.}
	\,\,
	David Eppstein,\!\!%
	\thanks{Department of Computer Science, University of California, Irvine, California, USA (\texttt{eppstein@uci.edu}).}
	\,\,
	Robert Hickingbotham,\!\!%
	\thanks{School of Mathematics, Monash University, Melbourne, Australia (\texttt{robert.hickingbotham@monash.edu}).}
	\,\,
	Pat Morin,\!\!%
	\thanks{School of Computer Science, Carleton University, Ottawa, Canada (\texttt{morin@scs.carleton.ca}). Research  supported by NSERC.}
	\,\,
	David R. Wood\thanks{School of Mathematics, Monash University, Melbourne, Australia (\texttt{david.wood@monash.edu}). Research supported by the Australian Research Council.}
}

\begin{document}
\maketitle

\begin{abstract}
We describe a family of graphs with queue-number at most 4 but unbounded stack-number. This resolves open problems of Heath, Leighton and Rosenberg (1992) and Blankenship and Oporowski (1999).
\end{abstract}

\bigskip

\section{Introduction}

Stacks and queues are fundamental data structures in computer science, but which is more powerful? In 1992, Heath, Leighton and Rosenberg~\cite{HLR92,HR92} introduced an approach for answering this question by defining the graph parameters \textit{stack-number} and \textit{queue-number} (defined below), which respectively measure the power of stacks and queues for representing graphs. The following fundamental questions, implicit in \citep{HLR92,HR92}, were made explicit by \citet{DujWoo05}\footnote{A \emph{graph parameter} is a function $\alpha$ such that $\alpha(G)\in\mathbb{R}$ for every graph $G$ and such that $\alpha(G_1)=\alpha(G_2)$ for all isomorphic graphs $G_1$ and $G_2$. A graph parameter $\alpha$ is \textit{bounded} by a graph parameter $\beta$ if there exists a function $f$ such that $\alpha(G) \leq f(\beta(G))$ for every graph $G$.}:
\begin{compactitem}
	\item Is stack-number bounded by queue-number?
	\item Is queue-number bounded by stack-number?
\end{compactitem}

If stack-number is bounded by queue-number but queue-number is not bounded by stack-number, then stacks would be considered to be more powerful than queues. Similarly, if the converse holds, then queues would be considered to be more powerful than stacks. Despite extensive research on stack- and queue-numbers, these questions have remained unsolved.


We now formally define stack- and queue-number. Let $G$ be a graph and let $\prec$ be a total order on $V(G)$.  Two disjoint edges $vw,xy\in E(G)$ with $v\prec w$ and $x\prec y$ \emph{cross} with respect to $\prec$ if $v\prec x\prec w\prec y$ or $x\prec v\prec y\prec w$, and \emph{nest} with respect to $\prec$ if $v\prec x\prec y\prec w$ or $x\prec v\prec w\prec y$. Consider a function $\varphi:E(G)\to\{1,\ldots,k\}$ for some $k\in\N$. Then $(\prec,\varphi)$ is a \emph{$k$-stack layout} of $G$ if $vw$ and $xy$ do not cross for all edges $vw,xy\in E(G)$ with $\varphi(vw) = \varphi(xy)$. Similarly, $(\prec,\varphi)$ is a \emph{$k$-queue layout} of $G$ if $vw$ and $xy$ do not nest for all edges $vw,xy\in E(G)$ with  $\varphi(vw)=\varphi(xy)$. See \cref{layouts} for examples. The smallest integer $s$ for which $G$ has an $s$-stack layout is called the \emph{stack-number} of $G$, denoted  $\sn(G)$. The smallest integer $q$ for which $G$ has a $q$-queue layout is called the \emph{queue-number} of $G$, denoted $\qn(G)$.

\begin{figure}[H]
	\centering
	\includegraphics{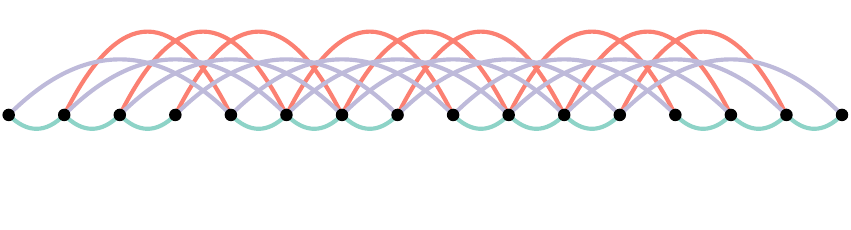}
	\includegraphics{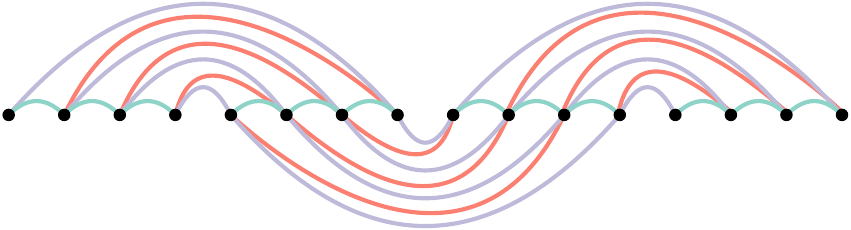}
	\caption{A $2$-queue layout and a $2$-stack layout of the triangulated grid graph $H_4$ defined below. Edges drawn above the vertices are assigned to the first queue/stack and edges drawn below the vertices are assigned to the second queue/stack.}
	\label{layouts}
\end{figure}

Given a $k$-stack layout $(\prec,\varphi)$ of a graph $G$, for each $i\in\{1,\dots,k\}$, the set $\varphi^{-1}(i)$ behaves like a stack, in the sense that each edge $vw \in \varphi^{-1}(i)$ with $v\prec w$ corresponds to an element in a sequence of stack operations, such that if we traverse the vertices in the order of $\prec$, then $vw$ is pushed onto the stack at $v$ and popped off the stack at $w$. Similarly, each set $\varphi^{-1}(i)$ in a queue layout  behaves like a queue. In this way, the stack-number and queue-number  respectively measure the power of stacks and queues to represent graphs.

Note that stack layouts are equivalent to book embeddings (first defined by \citet{Ollmann73} in 1973), and stack-number is also known as \emph{page-number}, \emph{book-thickness} or \emph{fixed outer-thickness}. Stack and queue layouts have other applications including computational complexity~\citep{GKS89,DSW16,Bourgain09,BY13},  RNA folding~\citep{HS99}, graph drawing in two~\citep{BB04,ADFPR12,SSSV-JGT96} and three dimensions~\citep{DMW05,Wood-GD01,DMW17,DPW04},
and fault-tolerant multiprocessing~\citep{CLR87,Rosenberg83a,Rosenberg86a,Rosenberg86}.
See \citep{BK79,Blankenship-PhD03,DujWoo04,DujWoo-DCG07,DJMMUW20,DFP13,BFGMMRU19,Yannakakis89,Yann20,MBKPRU20} for bounds on the stack- and queue-number for various graph classes.

\subsection*{Is Stack-Number Bounded by Queue-Number?}

This paper considers the first of the questions from the start of the paper. In a positive direction, \citet{HLR92}  showed that every 1-queue graph has a $2$-stack layout. On the other hand, they described graphs that need exponentially more stacks than queues. In particular, $n$-vertex ternary hypercubes have queue-number $O(\log n)$ and stack-number $\Omega(n^{1/9-\epsilon})$ for any $\epsilon>0$.

Our key contribution is the following theorem, which shows that stack-number is not bounded by queue-number.

\begin{thm}\label{family}
	For every $s\in \N$ there exists a graph $G$ with $\qn(G)\le 4$ and $\sn(G)>s$.
\end{thm}

This demonstrates that stacks are not more powerful than queues for representing graphs.

\subsection*{Cartesian Products}

As illustrated in \cref{graph}, the graph $G$ in \cref{family} is the cartesian product\footnote{For graphs $G_1$ and $G_2$, the \emph{cartesian product} $G_1\CartProd G_2$ is the graph with vertex set $\{(v_1,v_2): v_1 \in V(G_1), v_2 \in V(G_2)\}$, where $(v_1,v_2)(w_1,w_2)\in E(G_1\CartProd G_2)$ if $v_1=w_1$ and $v_2w_2\in E(G_2)$, or $v_1w_1\in E(G_1)$ and $v_2=w_2$. The \emph{strong product} $G_1\boxtimes G_2$ is the graph obtained from $G_1\CartProd G_2$ by adding the edge $(v_1,v_2)(w_1,w_2)$ whenever $v_1w_1\in E(G_1)$ and $v_2w_2\in E(G_2)$. Note that \citet{Pupyrev20} independently suggested using graph products to show that stack-number is not bounded by queue-number.} $S_b\CartProd H_n$ for sufficiently large $b$ and $n$, where $S_b$ is the star graph with root $r$ and $b$ leaves, and $H_n$ is the dual of the hexagonal grid, defined by
\begin{align*}
V(H_n)  :=\{1,\ldots,n\}^2 \quad \text{ and } \quad
E(H_n) & :=  \{(x,y)(x+1,y):x\in\{1,\ldots,n-1\},\,y\in\{1,\ldots,n\}\} \\
& \qquad \cup \{(x,y)(x,y+1):x\in\{1,\ldots,n\},\,y\in\{1,\ldots,n-1\}\} \\
& \qquad \cup \{(x,y)(x+1,y+1):x,y\in\{1,\ldots,n-1\}\} \enspace .
\end{align*}

\begin{figure}[H]
	\centering
	\begin{tabular}{m{.3\textwidth}m{2ex}m{.2\textwidth}m{2ex}m{.3\textwidth}}
		\includegraphics[width=.3\textwidth]{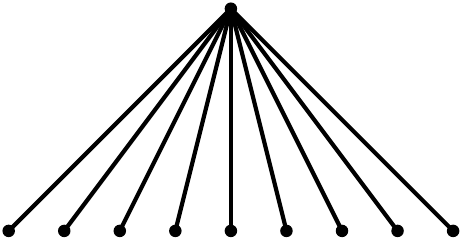} & $\CartProd$ & \includegraphics[width=.2\textwidth]{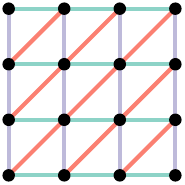} & $=$
		& \includegraphics[width=.3\textwidth]{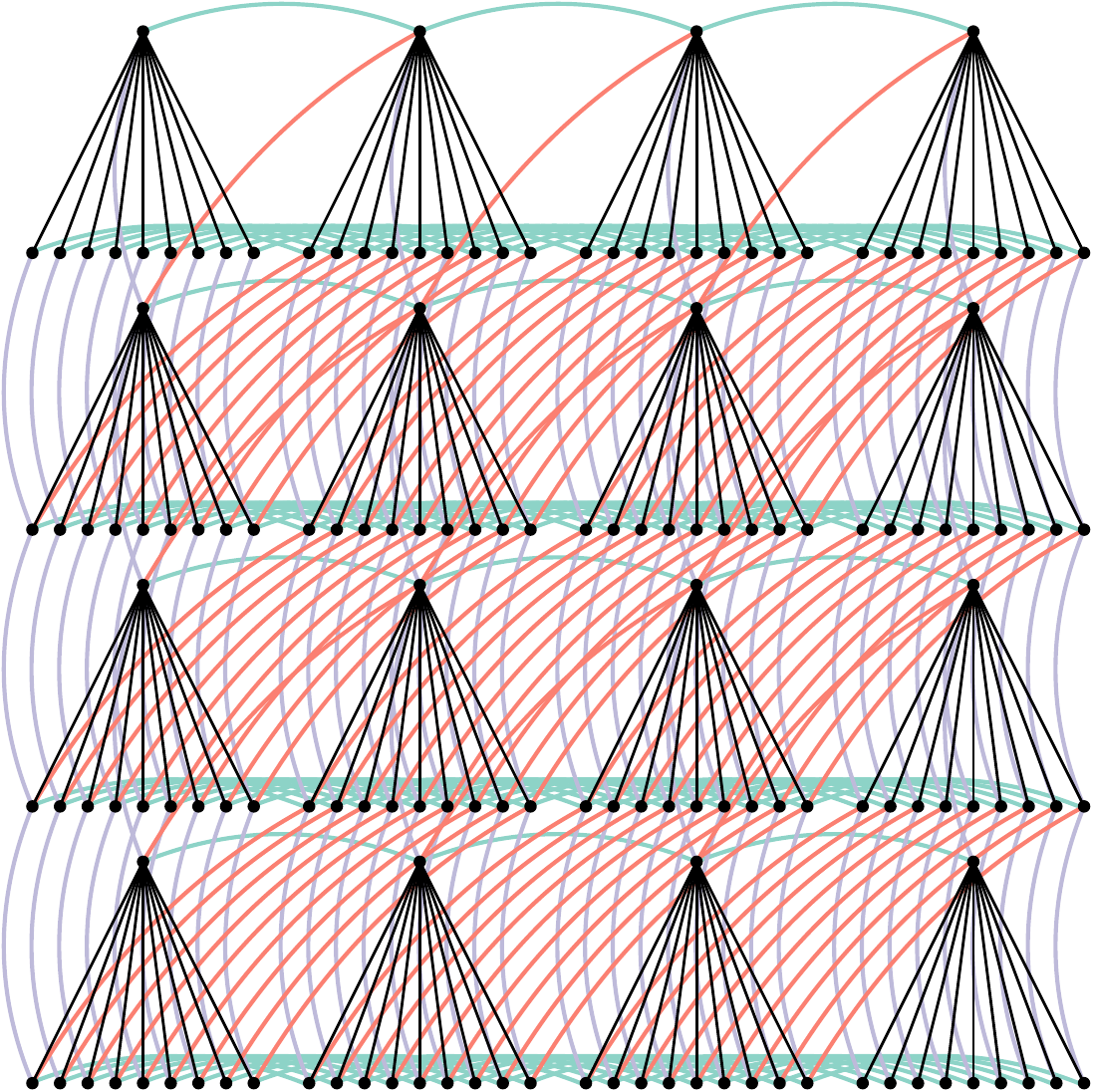}
	\end{tabular}
	\caption{$S_9 \CartProd H_4$.}
	\label{graph}
\end{figure}

We prove the following:

\begin{thm}
\label{Main}
For every $s\in\N$, if $b$ and $n$ are sufficiently large compared to $s$, then $$\sn(S_b\CartProd H_n) > s.$$
\end{thm}


We now show that $\qn(S_b\CartProd H_n)\leq 4$, which with \cref{Main} implies \cref{family}. We need the following definition due to \citet{Wood-Queue-DMTCS05}. A queue layout $(\varphi,\prec)$ is \emph{strict} if for every vertex $u\in V(G)$ and for all neighbours $v,w\in N_G(u)$, if $u\prec v\prec w$ or $v\prec w \prec u$, then $\varphi(uv)\neq \varphi(uw)$. Let $\sqn(G)$ be the minimum integer $k$ such that $G$ has a strict $k$-queue layout. To see that $\sqn(H_n) \leq 3$, order the vertices row-by-row and then left-to-right within a row, with vertical edges in one queue, horizontal edges in one queue, and diagonal edges in another queue. \citet{Wood-Queue-DMTCS05} proved that for all graphs $G_1$ and $G_2$,
\begin{equation}
\label{QueueProduct}
\qn(G_1 \CartProd G_2) \leq \qn(G_1) + \sqn(G_2).
\end{equation}
Of course, $S_b$ has a 1-queue layout (since no two edges are nested for any vertex-ordering). Thus $\qn(S_b \CartProd H_n)\leq 4$.

\citet{BK79} implicitly proved a result similar to \cref{QueueProduct} for stack layouts. Let $\dsn(G)$ be the minimum integer $k$ such that $G$ has a $k$-stack layout $(\prec,\varphi)$ where $\varphi$ is a proper edge-colouring of $G$; that is, $\varphi(vx)\neq\varphi(vy)$ for any two edges $vx,vy\in E(G)$ with a common endpoint. Then for every graph $G_1$ and every bipartite graph $G_2$,
\begin{equation}
\label{StackProduct}
\sn(G_1 \CartProd G_2) \leq \sn(G_1) + \dsn(G_2).
\end{equation}
The key difference between \cref{QueueProduct} and \cref{StackProduct} is that $G_2$ is assumed to be bipartite in \cref{StackProduct}. \cref{Main} says that this assumption is essential, since it is easily seen that $(\dsn(H_n))_{n\in\N}$ is bounded, but the stack number of $(S_b \CartProd H_n)_{b,n\in\N}$ is unbounded by \cref{Main}. We choose $H_n$ in \cref{Main} since it satisfies the Hex Lemma (\cref{hex_lemma} below), which quantifies the intuition that $H_n$ is far from being bipartite (while still having bounded queue-number and bounded maximum degree so that \cref{QueueProduct} is applicable). 


\subsection*{Subdivisions}

A noteworthy consequence of \Cref{family} is that it resolves a conjecture of \citet{BO99}. A graph $G'$ is a \textit{subdivision} of a graph $G$ if $G'$ can be obtained from $G$ by replacing the edges $vw$ of $G$ by internally disjoint paths $P_{vw}$ with endpoints $v$ and $w$. If each $P_{vw}$ has exactly $k$ internal vertices, then $G'$ is the \emph{$k$-subdivision} of $G$. If each $P_{vw}$ has at most $k$ internal vertices, then $G'$ is a \emph{$(\leq k)$-subdivision} of $G$. \citet{BO99} conjectured that the stack-number of $(\leq k)$-subdivisions ($k$ fixed)  is not much less than the stack-number of the original graph. More precisely:

\begin{conj}[\citep{BO99}]
\label{B_conj}
There exists a function $f$ such that for every graph $G$ and integer $k$, if $G'$ is any $(\leq k)$-subdivision of $G$, then $\sn(G) \leq f(\sn(G’),k)$.
\end{conj}

\citet{DujWoo05} established a connection between this conjecture and the question of whether stack-number is bounded by queue-number. In particular, they showed that if
\cref{B_conj} was true, then stack-number would be bounded by queue-number. Since \cref{family} shows that stack-number is not bounded by queue-number, \cref{B_conj} is false. The proof of \citet{DujWoo05} is based on the following key lemma: every graph $G$ has a $3$-stack subdivision with $1+2 \ceil{\log_2\qn(G)}$ division vertices per edge. Applying this result to the graph $G=S_b\CartProd H_n$ in \cref{family},
the $5$-subdivision of $S_b\CartProd H_n$ has a $3$-stack layout. If \cref{B_conj} was true, then  $\sn(S_b\CartProd H_n)$ would be at most $f( 3,5)$, contradicting \cref{family}.





\subsection*{Is Queue-Number Bounded by Stack-Number? }

It remains open whether queues are more powerful than stacks; that is, whether queue-number is bounded by stack-number. Several results are known about this problem. \citet{HLR92} showed that every 1-stack graph has a 2-queue layout. \citet{DJMMUW20} showed that planar graphs have bounded queue-number. (Note that graph products also feature heavily in this proof.)\ Since 2-stack graphs are planar, this implies that 2-stack graphs have bounded queue-number. It is open whether 3-stack graphs have bounded queue-number. In fact, the case of three stacks is as hard as the general question. \citet{DujWoo05} proved that queue-number is bounded by stack-number if and only if 3-stack graphs have bounded queue-number. Moreover, if this is true then queue-number is bounded by a polynomial function of stack-number.

\section{Proof of \cref{Main}}

We now turn to the proof of our main result, the lower bound on $\sn(G)$, where $G:= S_b\CartProd H_n$. Consider a hypothetical $s$-stack layout $(\varphi,\prec)$ of $G$ where $n$ and $b$ are chosen sufficiently large compared to $s$ as detailed below. We begin with three lemmas that, for sufficiently large $b$, provide a large sub-star $S_d$ of $S_b$ for which the induced stack layout of $S_d\CartProd H_n$ is highly structured.

For each node $v$ of $S_b$, define $\pi_v$ as the permutation of $\{1,\ldots,n\}^2$ in which $(x_1,y_1)$ appears before $(x_2,y_2)$ if and only if $(v,(x_1,y_1))\prec (v,(x_2,y_2))$. The following lemma is an immediate consequence of the Pigeonhole Principle:

\begin{lem}\label{uniform_order}
    There exists a permutation $\pi$ of $\{1,\ldots,n\}^2$ and a set $L_1$ of leaves of $S_b$ of size $a\ge b/(n^2)!$ such that $\pi_{v}=\pi$ for each $v\in L_1$.
\end{lem}


For each leaf $v$ in $L_1$, let $\varphi_v$ be the edge colouring of $H_n$ defined by $\varphi_v(xy):=\varphi((v,x)(v,y))$ for each $xy\in E(H_n)$. Since $H_n$ has maximum degree $6$ and is not 6-regular, it has fewer than $3n^2$ edges.  Therefore there are fewer than $s^{3n^2}$ edge colourings of $H_n$ using $s$ colours.  Another application of the Pigeonhole Principle proves the following:

\begin{lem}\label{uniform_colour}
    There exists a subset $L_2\subseteq L_1$ of size $c\ge a/s^{3n^2}$
    and an edge colouring $\phi:E(H_n)\to\{1,\ldots,s\}$ such that $\varphi_v=\phi$ for each $v\in L_2$.
\end{lem}

Let $S_{c}$ be the subgraph of $S_b$ induced by $L_2\cup\{r\}$. The preceding two lemmas ensure that, for distinct leaves $v$ and $w$ of $S_{c}$, the stack layouts of the isomorphic graphs $G[\{(v,p):p\in V(H_n)\}]$ and $G[\{(w,p):p\in V(H_n)\}]$ are identical. The next lemma is a statement about the relationship between the stack layouts of $G[\{(v,p):v\in V(S_{c})\}]$ and $G[\{(v,q):v\in V(S_{c})\}]$ for  distinct $p,q\in V(H_n)$.  It does not assert that these two layouts are identical but it does state that they fall into one of two categories.

\begin{lem}\label{forward_or_backward}
    There exists a sequence $u_1,\ldots,u_{d}\in L_2$ of length $d\ge c^{1/2^{n^2-1}}$ such that, for each $p\in V(H_n)$, either  $(u_1,p)\prec (u_2,p)\prec\cdots\prec (u_{d},p)$ or $(u_1,p)\succ (u_2,p)\succ\cdots\succ (u_{d},p)$.
\end{lem}

\begin{proof}
    Let $p_1,\ldots,p_{n^2}$ denote the vertices of $H_n$ in any order.
    Begin with the sequence $V_1:=v_{1,1},\ldots,v_{1,c}$ that contains all $c$ elements of $L_2$ ordered so that $(v_{1,1},p_1)\prec\cdots\prec(v_{1,c},p_1)$.  For each $i\in\{2,\ldots,n^2\}$, the Erd\H{o}s-Szekeres Theorem~\citep{ES35} implies that $V_{i-1}$ contains a subsequence $V_i:=v_{i,1},\ldots,v_{i,|V_i|}$ of length $|V_i|\ge \sqrt{|V_{i-1}|}$ such that $(v_{i,1},p_i)\prec\cdots\prec(v_{i,|V_i|},p_i)$ or $(v_{i,1},p_i)\succ\cdots\succ(v_{i,|V_i|},p_i)$.  It is straightforward to verify by induction on $i$ that $|V_i| \ge c^{1/2^{i-1}}$ resulting in a final sequence $V_{n^2}$ of length at least $c^{1/2^{n^2-1}}$.    
\end{proof}

For the rest of the proof we work with the star $S_d$ whose leaves are $u_1,\ldots,u_d$ described in \cref{forward_or_backward}.  Consider the (improper) colouring of $H_n$ obtained by colouring each vertex $p\in V(H_n)$ \emph{red} if $(u_1,p)\prec\cdots\prec (u_d,p)$ and colouring $p$ \emph{blue} if $(u_1,p)\succ\cdots\succ(u_d,p)$. We need the following famous Hex Lemma~\citep{Gale79}.

\begin{lem}[\citep{Gale79}] \label{hex_lemma}
Every vertex 2-colouring of $H_n$ contains a monochromatic path on $n$ vertices.
\end{lem}


Apply \cref{hex_lemma} with the above-defined colouring of $H_n$. We obtain a path subgraph  $P=(p_1,\ldots,p_n)$ of $H_n$ that, without loss of generality, consists entirely of red vertices; thus $(u_1,p_j)\prec\cdots\prec (u_d,p_j)$ for each $j\in\{1,\ldots,n\}$.  Let $X$ be the subgraph $S_d \CartProd P$ of $G$.

\begin{lem}\label{twister}
$X$ contains a set of at least $\min\{\lfloor d/2^{n}\rfloor,\lceil n/2\rceil\}$ pairwise crossing edges with respect to $\prec$.
\end{lem}

\begin{proof}
	Extend the total order $\prec$ to a partial order over subsets of $V(G)$, where for all $V,W\subseteq V(G)$, we have $V\prec W$ if and only if $v\prec w$ for each $v\in V$ and each $w\in W$.  We abuse notation slightly by using $\prec$ to compare elements of $V(G)$ and subsets of $V(G)$ so that, for $v\in V(G)$ and $V\subseteq V(G)$, $v\prec V$ denotes $\{v\}\prec V$.
    We will define sets $A_1\supseteq \cdots\supseteq A_{n}$ of leaves of $S_d$ so that each $A_i$ satisifies the following conditions:
    \begin{compactenum}[(C1)]
        \item $A_i$ contains $d_i\ge d/2^{i-1}$ leaves of $S_d$.
        \item Each leaf $v\in A_i$ defines an $i$-element vertex set $Z_{i,v}:=\{(v,p_j):j\in\{1,\ldots,i\}\}$.  For any distinct $v,w\in A_i$, the sets $Z_{i,v}$ and $Z_{i,w}$ are \emph{separated} with respect to $\prec$; that is, $Z_{i,v}\prec Z_{i,w}$ or $Z_{i,v}\succ Z_{i,w}$.
    \end{compactenum}

    Before defining $A_1,\ldots,A_n$ we first show how the existence of the set $A_n$ implies the lemma.  To avoid triple-subscripts, let $d':=d_n\ge d/2^{n-1}$. By (C2), the set $A_n$ defines vertex sets $Z_{n,v_1}\prec\cdots\prec Z_{n,v_{d'}}$ (see \cref{fig_twister}). The root $r$ of $S_b$ is adjacent to each of $v_{1},\ldots,v_{d'}$ in $S_d$. Thus, for each $j\in\{1,\ldots,n\}$ and each $i\in\{1,\ldots,d'\}$, the edge $(r,p_j)(v_i,p_j)$ is in $X$. Hence, $(r,p_j)$ is adjacent to an element of each of $Z_{n,v_1},\ldots,Z_{n,v_{d'}}$.
	\begin{figure}[!h]
		\centering\includegraphics{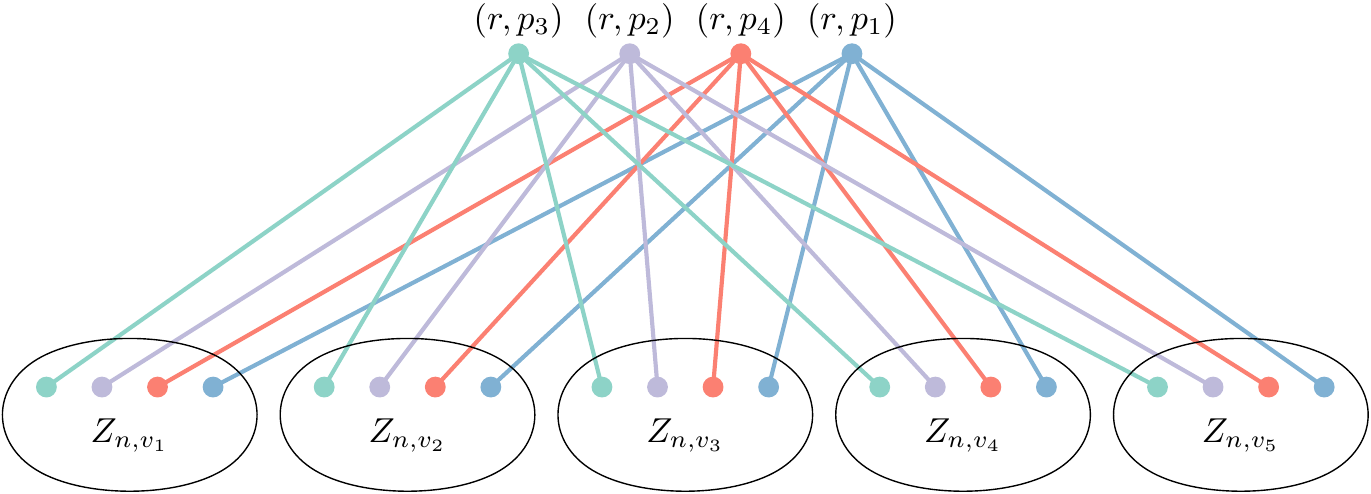}
		\caption{The sets $Z_{n,v_1},\ldots,Z_{n,v_{d'}}$ where $n=4$ and $d'=5$.}
		\label{fig_twister}
	\end{figure}

   Since $Z_{n,v_1},\ldots,Z_{n,v_{d'}}$ are separated with respect to $\prec$,
   if we imagine identifying the vertices in each set $Z_{n,v_i}$, this situation looks like a complete bipartite graph $K_{n,d'}$ with the root vertices $L:=\{(r,p_j):j\in\{1,\ldots,n\}\}$ in one part and the groups $R:=Z_{n,v_1}\cup\cdots\cup Z_{n,v_{d'}}$ in the other part.  Any linear ordering of $K_{n,d'}$ has a large set of pairwise crossing edges. So, intuitively, the induced subgraph $X[L\cup R]$ should also have a large set of pairwise crossing edges. 
   
   We formalize this idea as follows: Label the vertices in $L$ as $r_1,\ldots,r_n$ so that $r_1\prec \cdots\prec r_{n}$.  Then at least one of the following two cases applies (see \figref{median}):
    \begin{enumerate}
        \item $Z_{n,\lfloor d'/2\rfloor}\prec r_{\lceil n/2\rceil}$ in which case the graph between $r_{\lceil n/2\rceil},\ldots,r_{n}$ and $Z_{n,1},\ldots,Z_{n,\lfloor d'/2\rfloor}$ has a set of at least $\min\{\lfloor d'/2\rfloor,\lceil n/2\rceil\}$ pairwise-crossing edges.
        \item $r_{\lceil n/2\rceil}\prec Z_{\lceil d'/2\rceil+1}$ in which case the graph between $r_1,\ldots,r_{\lceil n/2\rceil}$ and $Z_{\lceil d'/2\rceil+1},\ldots,Z_{d'}$ has a set of $\min\{\lfloor d'/2\rfloor,\lceil n/2\rceil\}$ pairwise-crossing edges.
    \end{enumerate}
	\begin{figure}[!h]
		\centering
			\includegraphics{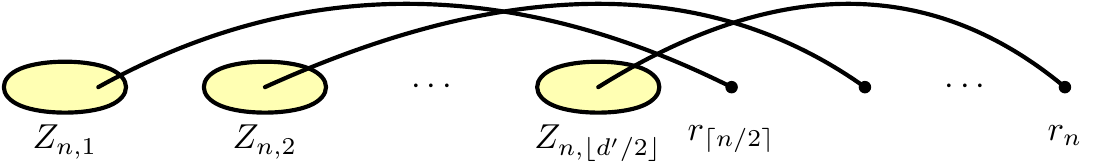} \\[2em]
			\includegraphics{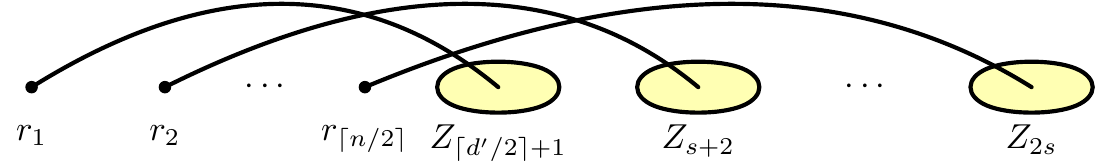}
		\caption{The two cases in the proof of \cref{twister}.}
		\figlabel{median}
	\end{figure}
	Since, by (C1), $d'\ge d/2^{n-1}$, either case results in a set of pairwise-crossing edges of size at least $\min\{\lfloor d/2^n\rfloor,\lceil n/2\rceil\}$, as claimed.
	
    It remains to define the sets $A_1\supseteq\cdots\supseteq A_n$ that satisfy (C1) and (C2).  Let $A_1$ be the set of all the leaves of $S_d$.  For each $i\in\{2,\ldots,n\}$, assuming that $A_{i-1}$ is already defined, the set $A_i$ is defined as follows: For brevity, let $m:= |A_{i-1}|$. Let $Z_1,\ldots,Z_m$ denote the sets $Z_{i-1, v}$ for each $v\in A_{i-1}$ ordered so that $Z_1\prec\cdots\prec Z_m$. By Property (C2), this is always possible. Label the vertices of $A_{i-1}$ as $v_1,\ldots,v_m$ so that $(v_1,p_{i-1})\prec\cdots\prec (v_m  ,p_{i-1})$.   (This is equivalent to naming them so that $(v_j,p_{i-1})\in Z_j$ for each $j\in\{1,\ldots,m\}$.)  Define the set $A_i:=\{v_{2k+1}:k\in\{0,\ldots,\lfloor(m-1)/2\rfloor\}\}=\{v_{j}\in A_{i-1}:\text{$j$ is odd}\}$.  This completes the definition of $A_1,\ldots,A_n$.

	We now verify that $A_i$ satisfies (C1) and (C2) for each $i\in\{1,\ldots,n\}$.  We do this by induction on $i$. The base case $i=1$ is trivial, so now assume that $i\in\{2,\ldots,n\}$.   To see that $A_i$ satisfies (C1) observe that $|A_i|=\lceil |A_{i-1}|/2\rceil \ge |A_{i-1}|/2\ge d/2^{i-1}$, where the final inequality follows by applying the inductive hypothesis $|A_{i-1}|\ge d/2^{i-2}$.  Now it remains to show that $A_i$ satisfies (C2). Again, let $m:= |A_{i-1}|$. 

    Recall that, for each $v\in A_{i-1}$, the edge $e_v:=(v,p_{i-1})(v,p_i)$ is in $X$.  We have the following properties:
    \begin{compactenum}[(P1)]
        \item By \cref{uniform_colour}, $\varphi(e_v)=\phi(p_{i-1}p_i)$ for each $v\in A_{i-1}$, 
        \item Since $p_{i-1}$ and $p_i$ are both red,  for each $v,w\in A_{i-1}$, we have $(v,p_{i-1})\prec (w,p_{i-1})$ if and only if $(v,p_{i})\prec (w,p_{i})$.
        \item By \cref{uniform_order}, $(v,p_{i-1})\prec (v,p_i)$ for every $v\in A_{i-1}$ or $(v,p_{i-1})\succ (v,p_i)$ for every $v\in A_{i-1}$.
    \end{compactenum}
We claim that these three conditions imply that the vertex sets $\{(v,p_{i-1}):v\in A_{i-1}\}$ and $\{(v,p_i):v\in A_{i-1}\}$ interleave perfectly with respect to $\prec$.     More precisely:
    
	\begin{clm}\clmlabel{interleave} 
		$(v_1,p_{i-1+t})\prec (v_1,p_{i-t}) \prec (v_2,p_{i-1+t}) \prec (v_2,p_{i-t}) \cdots \prec (v_m,p_{i-1+t}) \prec (v_m,p_{i-t})$ 
		for some $t\in\{0,1\}$.
	\end{clm}
	
	\begin{proof}[Proof of \clmref{interleave}]
		By (P3) we may assume, without loss of generality, that $(v,p_{i-1})\prec (v,p_i)$ for each $v\in A_{i-1}$, in which case we are trying to prove the claim for $t=0$.  Therefore, it is sufficient to show that $(v_j,p_i)\prec (v_{j+1},p_{i-1})$ for each $j\in\{1,\ldots,m-1\}$.  For the sake of contradiction, suppose $(v_j,p_{i})\succ (v_{j+1},p_{i-1})$ for some $j\in\{1,\ldots,m-1\}$. By the labelling of $A_{i-1}$,  $(v_j,p_{i-1})\prec (v_{j+1},p_{i-1})$ so, by (P2),  $(v_{j},p_i) \prec (v_{j+1},p_i)$.  Therefore
		\[
			(v_j,p_{i-1})\prec (v_{j+1},p_{i-1})\prec(v_{j},p_i) \prec
		   (v_{j+1}, p_i) \enspace .
	   	\]
		Therefore the edges $e_{v_j}=(v_j,p_{i-1})(v_j,p_{i})$ and $e_{v_{j+1}}=(v_{j+1},p_{i-1})(v_{j+1},p_i)$ cross with respect to $\prec$.  But this is a contradiction since, by (P1),  $\varphi(e_{v_j}) =\varphi(e_{v_{j+1}})=\phi(p_{i-1}p_i)$.
		This contradiction completes the proof of \clmref{interleave}.
	\end{proof}

We now complete the proof that $A_i$ satisfies (C2). Apply \clmref{interleave} and assume without loss of generality that $t=0$, so that
	\[
		(v_1,p_{i-1})\prec (v_1,p_{i}) \prec (v_2,p_{i-1}) \prec (v_2,p_{i}) \cdots \prec (v_m,p_{i-1}) \prec (v_m,p_{i}) \enspace .
	\]

    For each $j\in\{1,\ldots,m-2\}$, we have $(v_{j+1},p_{i-1})\in Z_{j+1}\prec Z_{j+2}$, so  $(v_j,p_i)\prec (v_{j+1},p_{i-1}) \prec Z_{j+2}$.  Therefore $Z_j\cup\{(v_j,p_i)\} \prec Z_{j+2}$.  By a symmetric argument, $Z_j\cup\{(v_j,p_i)\} \succ Z_{j-2}$ for each  $j\in\{3,\ldots,m\}$.  Finally, since $(v_{j},p_i)\prec (v_{j+2},p_i)$ for each odd $i\in\{1,\ldots,m\}$, we have $Z_{j}\cup\{(v_j,p_i)\} \prec Z_{j+2}\cup\{(v_{j+2},p_i)\}$ for each odd $j\in\{1,\ldots,m-2\}$.  Thus $A_i$ satisfies (C2) since the sets $Z_1\cup\{(v_1,p_i)\},Z_3\cup\{(v_3,p_i)\},\ldots,Z_{2\lfloor (m-1)/2\rfloor+1} \cup (v_{2\lfloor (m-1)/2\rfloor+1},p_i)$ are precisely the sets $Z_{i,1},\ldots,Z_{i,d_i}$ determined by our choice of $A_i$.
\end{proof}

%
%

\begin{proof}[Proof of \cref{Main}]
Let $G := S_b \CartProd H_n$, where $n :=2s+1$ and $b := (n^2)!\, s^{3n^2}\, ((s+1)2^n)^{2^{n^2-1}} $. Suppose that $G$ has an $s$-stack layout  $(\varphi,\prec)$. In particular, there are no $s+1$ pairwise crossing edges in $G$ with respect to $\prec$. By \cref{uniform_colour,uniform_order,forward_or_backward}, we have $a\ge b/(n^2)! = s^{3n^2}\, ((s+1)2^n)^{2^{n^2-1}}$ and $c\ge a/s^{3n^2} \geq ((s+1)2^n)^{2^{n^2-1}}$ and
$d\ge c^{1/2^{n^2-1}} \ge (s+1)2^n$. By \cref{twister}, the graph $X$, which is a subgraph of $G$, contains $\min\{\lfloor d/2^{n}\rfloor,\lceil n/2\rceil\}=s+1$ pairwise crossing edges with respect to $\prec$. This contradictions shows that $\sn(G)> s$.
\end{proof}

\section{Reflections}

We now mention some further consequences and open problems that arise from our main result. 

\citet{NOW11} proved that graph classes with bounded stack-number or bounded queue-number have bounded expansion; see \citep{Sparsity} for background on bounded expansion classes. The converse is not true, since cubic graphs (for example) have bounded expansion, unbounded stack-number~\citep{Malitz94a} and unbounded queue-number~\citep{Wood-QueueDegree}. However, prior to the present work it was open whether graph classes with polynomial expansion have bounded stack-number or bounded queue-number. It follows from the work of \citet[Theorem~19]{DHJLW21} that $(S_b \CartProd H_n)_{b,n\in\N}$ has polynomial expansion. So \cref{Main} implies there is a class of graphs with polynomial expansion and with unbounded stack-number. It remains open whether graph classes with polynomial expansion have bounded queue-number. See \citep{DJMMUW20,DMW} for several examples of graph classes with polynomial expansion and bounded queue-number.

Our main result also resolves a question of \citet{BGKTWb} concerning \emph{sparse twin-width}; see \citep{BKTW,BGKTWa,BGKTWb} for the definition and background on (sparse) twin-width. \citet{BGKTWb} proved that graphs with bounded stack-number have bounded sparse twin-width, and they write that they ``believe that the inclusion is strict''; that is, there exists a class of graphs with bounded sparse twin-width and unbounded stack-number. \Cref{Main} confirms this intuition, since the class of all subgraphs of $(S_b \CartProd H_n)_{b,n\in\N}$ has bounded sparse twin-width (since \citet{BGKTWb} showed that any hereditary class of graphs with bounded queue-number has bounded sparse twin-width). 
It remains open whether bounded sparse twin-width coincides with bounded queue-number.

Finally, we mention some more open problems: 

\begin{compactitem}
\item Recall that every 1-queue graph has a 2-stack layout \citep{HLR92} and we proved that there are 4-queue graphs with unbounded stack-number. The following questions remain open: Do 2-queue graphs have bounded stack-number? Do 3-queue graphs have bounded stack-number? 

\item Since $H_n\subseteq P \boxtimes P$ where $P$ is the $n$-vertex path, \cref{family} implies that $\sn(S\boxtimes P\boxtimes P)$ is unbounded for stars $S$ and paths $P$. It is easily seen that $\sn(S\boxtimes P)$ is bounded~\citep{Pupyrev20}. The following question naturally arises (independently asked by \citet{Pupyrev20}): Is $\sn(T \boxtimes P)$ bounded for all trees $T$ and paths $P$? We conjecture the answer is ``no''.

\end{compactitem}

\subsubsection*{Acknowledgement} Thanks to \'Edouard Bonnet for a helpful comment. 

{
\fontsize{11.3}{12.5}
\selectfont 

\let\oldthebibliography=\thebibliography
\let\endoldthebibliography=\endthebibliography
\renewenvironment{thebibliography}[1]{%
\begin{oldthebibliography}{#1}%
\setlength{\parskip}{0.2ex}%
\setlength{\itemsep}{0.2ex}%
}{\end{oldthebibliography}}


\def\soft#1{\leavevmode\setbox0=\hbox{h}\dimen7=\ht0\advance \dimen7
	by-1ex\relax\if t#1\relax\rlap{\raise.6\dimen7
		\hbox{\kern.3ex\char'47}}#1\relax\else\if T#1\relax
	\rlap{\raise.5\dimen7\hbox{\kern1.3ex\char'47}}#1\relax \else\if
	d#1\relax\rlap{\raise.5\dimen7\hbox{\kern.9ex \char'47}}#1\relax\else\if
	D#1\relax\rlap{\raise.5\dimen7 \hbox{\kern1.4ex\char'47}}#1\relax\else\if
	l#1\relax \rlap{\raise.5\dimen7\hbox{\kern.4ex\char'47}}#1\relax \else\if
	L#1\relax\rlap{\raise.5\dimen7\hbox{\kern.7ex
			\char'47}}#1\relax\else\message{accent \string\soft \space #1 not
		defined!}#1\relax\fi\fi\fi\fi\fi\fi}

}

\end{document}